\documentclass[12pt,reqno]{amsart}
\usepackage{amssymb}
\usepackage[mathscr]{eucal}

\usepackage{amssymb}

\def\PP{{\mathbb P}} 

\def\Acal{{\mathcal A}} 
 
\def\Mcal{{\mathcal M}} 
\def\Ocal{{\mathcal O}}

\def\Pcal{{\mathcal P}}

\def\Xcal{{\mathcal X}}

\newcommand\Mod{\operatorname{Mod}}
\newcommand\sym{\operatorname{Sym}}

\newtheorem{theorem}{Theorem}[section]
\newtheorem{proposition}[theorem]{Proposition}
\newtheorem{lemma}[theorem]{Lemma}

\theoremstyle{definition}

\numberwithin{equation}{section}

\usepackage[
	hypertexnames=false,
	hyperindex,
	pagebackref,
	pdftex,
	breaklinks=true,
	bookmarks=false,
	colorlinks,
	linkcolor=blue,
	citecolor=red,
	urlcolor=red,
]{hyperref}

\begin{document}
\baselineskip=15.5pt

\title[Moduli space of projective structures on complex curves]{Functions on the moduli
space of projective structures on complex curves}

\author[I. Biswas]{Indranil Biswas}

\address{Department of Mathematics, Shiv Nadar University, NH91, Tehsil
Dadri, Greater Noida, Uttar Pradesh 201314, India}

\email{indranil.biswas@snu.edu.in, indranil29@gmail.com}



\date{}

\begin{abstract}
We investigate the moduli space ${\mathcal P}_g$ of smooth complex projective curves of
genus $g$ equipped with a projective structure. When $g\, \geq\, 3$, it is shown that this moduli space
${\mathcal P}_g$ does not admit any nonconstant algebraic function. This is in contrast with the case
of ${\mathcal P}_2$ which is known to be an affine variety.
\end{abstract}

\maketitle

\section{Introduction}

Given a $C^\infty$ oriented surface $S$, a projective structure on it is defined by giving a coordinate atlas, where each
coordinate function has target ${\mathbb C}{\mathbb P}^1$, such that all the transition functions are M\"obius transformations
(same as holomorphic automorphisms of ${\mathbb C}{\mathbb P}^1$). Therefore, with all the transition functions being
holomorphic, a projective structure on $S$ produces a complex structure on $S$. However, projective structures are
much finer than complex structures. In fact, while the moduli space of compact Riemann surfaces of genus $g\,\geq\,2$ has complex
dimension $3g-3$, the moduli space of genus $g$ surfaces with a projective structure is of complex dimension $6g-6$. When
$g\, =\, 1$, these two numbers are $1$ and $2$ respectively.

For any $g\, \geq\, 2$, the moduli space ${\mathcal M}_g$ parametrizing compact Riemann surfaces of genus $g$ is a smooth
quasiprojective orbifold. The moduli space ${\mathcal P}_g$ parametrizing genus $g$ surfaces
with a projective structure is an algebraic torsor over ${\mathcal M}_g$ for the cotangent bundle $T^*{\mathcal M}_g$. The main
result proved here says that ${\mathcal P}_g$ does not admit any nonconstant algebraic function if
$g\, \geq\, 3$ (see Theorem \ref{thm1}). Since ${\mathcal P}_2$ is an affine variety, it admits plenty of
nonconstant algebraic functions.

\section{The moduli space of projective structures}

\subsection{Projective structure}

Let $X$ be a connected Riemann surface. A holomorphic coordinate chart on $X$ is a pair
of the form $(U,\, \phi)$, where $U\, \subset\, X$ is an analytic open subset and $\phi\,
:\, U\, \longrightarrow\, {\mathbb C}{\mathbb P}^1$ is a holomorphic embedding. A holomorphic
coordinate atlas on $X$ is a collection of coordinate charts $\{(U_i,\, \phi_i)\}_{i\in I}$
such that $X\,=\, \bigcup_{i\in I} U_i$. A projective structure on $X$ is given by a holomorphic
coordinate atlas $\{(U_i,\, \varphi_i)\}_{i\in I}$ satisfying the condition that
for every $i,\, j\, \in\, I\times I$ with $U_i\bigcap U_j\, \not=\, \emptyset$, and every connected
component $V_c\, \subset\, U_i\bigcap U_j$, there is an element
$A^c_{j,i} \, \in\, \text{Aut}({\mathbb C}{\mathbb P}^1)
\,=\, \text{PSL}(2, {\mathbb C})$ such that the
map $(\phi_j\circ\phi^{-1}_i)\big\vert_{\phi_i(V_c)}$ is the restriction of the automorphism $A^c_{j,i}$ of
${\mathbb C}{\mathbb P}^1$ to $\phi_i(V_c)$.

Two holomorphic coordinate atlases $\{(U_i,\, \phi_i)\}_{i\in I}$ and $\{(U_i,\, \phi_i)\}_{i\in I'}$
satisfying the above condition on $\phi_j\circ\phi^{-1}_i$ are called {\it equivalent} if their union
$\{(U_i,\, \phi_i)\}_{i\in I\cup
I'}$ also satisfies the above condition. A {\it projective structure} on $X$ is an equivalence class of
holomorphic coordinate atlases satisfying the above condition (see \cite{Gu}). Every Riemann surface admits
a projective structure. In fact, the uniformization theorem says that the universal cover of $X$ is one of
$\mathbb C$, ${\mathbb C}{\mathbb P}^1$ and the upper-half plane $\mathbb H$. The holomorphic automorphism group
of each of them is contained in $\text{Aut}({\mathbb C}{\mathbb P}^1)$. Therefore, the tautological projective
structure on the universal cover of $X$ produces a projective structure on $X$. Note that the projective structure
on $X$ obtained this way does not depend on the identification of the universal cover of $X$ with
$\mathbb C$, ${\mathbb C}{\mathbb P}^1$ or $\mathbb H$.

The above definition of a projective structure does not bring out its algebraic geometric aspect.
We will recall another formulation of the definition of a projective structure.

Let $X$ be an irreducible smooth complex projective curve. Fix the Borel subgroup
\begin{equation}\label{bsg}
B\, :=\, \Big\{
\begin{pmatrix}
a & b\\
c &d
\end{pmatrix}
\, \in\, \text{PSL}(2,{\mathbb C})\, \mid\, b\, =\, 0\Big\}\, \subset\, \text{PSL}(2,{\mathbb C})\, .
\end{equation}
Let $P\, \longrightarrow\, X$ be an algebraic principal $\text{PSL}(2,{\mathbb C})$--bundle on $X$
equipped with an algebraic connection $D$. We note that any algebraic connection on $X$ is integrable
because $\Omega^{2,0}_X\,=\, 0$. Let
\begin{equation}\label{l1}
P_B\, \subset\, P
\end{equation}
be an algebraic reduction of structure group, to the subgroup $B$ in \eqref{bsg}, given by an algebraic section
\begin{equation}\label{e1}
\sigma\, :\, X\, \longrightarrow\, P/B
\end{equation}
of the natural projection $f\, :\, P/B\, \longrightarrow\, X$. The connection $D$ on $P$ induces a connection on the
fiber bundle $f$, which, in turn, decomposes the tangent bundle $T(P/B)$ as
\begin{equation}\label{j3}
T(P/B)\,=\, T_f\oplus {\mathcal H}\, ,
\end{equation}
where $T_f\, \subset\, T(P/B)$ is the relative tangent bundle for the projection $f$ and ${\mathcal H}\, \subset\,
T(P/B)$ is the horizontal subbundle for the connection on $P/B$. Let
\begin{equation}\label{e2}
\widetilde{\sigma}\, :\, TX\, \longrightarrow\, \sigma^*T_f
\end{equation}
be the composition of homomorphisms 
$$
TX\, \stackrel{d\sigma}{\longrightarrow}\,\sigma^* T(P/B)\,=\, \sigma^*T_f\oplus \sigma^*{\mathcal H}
\, \longrightarrow\, \sigma^*T_f
$$
constructed using the decomposition in \eqref{j3}, where $d\sigma$ is the differential of $\sigma$ in \eqref{e1}. This
homomorphism $\widetilde{\sigma}$ is the second fundamental form of the reduction $P_B$ in \eqref{l1}
for the connection $D$.

A projective structure on $X$ is a triple $(P,\, P_B,\, D)$ as above such that the homomorphism
$\widetilde{\sigma}$ in \eqref{e2} is an isomorphism \cite{Gu}.

The holomorphic cotangent bundle of $X$ will be denoted by $K_X$. The extensions of ${\mathcal O}_X$ by $K_X$ are
parametrized by $$H^1(X,\, \text{Hom}({\mathcal O}_X,\, K_X))\,=\, H^1(X,\, K_X)\,=\, \mathbb C.$$
Let
$$
0\, \longrightarrow\, K_X\, \longrightarrow\, V \, \longrightarrow\, {\mathcal O}_X
\, \longrightarrow\, 0
$$
be the unique extension corresponding to $1\, \in\,  H^1(X,\, \text{Hom}({\mathcal O}_X,\, K_X))\,=\, \mathbb C$.
We note that $V\otimes TX\,=\, J^1(TX)$ if $\text{genus}(X)\,\not=\, 1$.

We assume that $\text{genus}(X)\,\not=\, 1$.

Let
\begin{equation}\label{e3}
P^0\, \longrightarrow\, X
\end{equation}
be the principal $\text{PSL}(2,{\mathbb C})$--bundle given by ${\mathbb P}(V)$; the fiber
of $P^0$ over any $x\, \in\, X$ is the space of all isomorphisms ${\mathbb C}{\mathbb P}^1\,
\longrightarrow\, P^0_x$. The above line subbundle $K_X\, \hookrightarrow\, V$ produces a reduction
of structure group
\begin{equation}\label{f1}
P^0_B\, \subset\, P^0
\end{equation}
to the Borel subgroup $B\,\subset\, \text{PSL}(2,{\mathbb C})$ in \eqref{bsg}. The fiber of $P^0_B$ over any
$x\, \in\, X$ is the space of all isomorphisms ${\mathbb C}{\mathbb P}^1\,
\longrightarrow\, P^0_x$ that take the point $e_2\, \in\, {\mathbb C}{\mathbb P}^1$ to the point of $P^0_x$ given by 
$(K_X)_x\, \subset\, V_x$, where $\{e_1,\, e_2\}$ is the standard basis of ${\mathbb C}^2$.

For any projective structure $(P,\, P_B,\, D)$ on $X$, the pair $(P,\, P_B)$ coincides with
$(P^0,\, P^0_B)$ \cite{Gu}. So only the connection $D$ moves when the projective structure on the
Riemann surface $X$
moves. Furthermore, for every connection $D$ on $P^0$, the triple $(P^0,\, P^0_B,\, D)$ is actually
a projective structure on $X$. From these it follows that the space of all projective structures
on $X$ is an affine space modeled on the vector space $H^0(X,\, K^{\otimes 2}_X)$. In particular, ${\mathbb C}{\mathbb P}^1$
has exactly one projective structure.

\subsection{Moduli spaces}

Fix a compact connected oriented $C^\infty$ surface $\Sigma$ of genus $g\,\ge\, 2$. Denote by ${\rm Diff}^+(\Sigma)$ the group
of all orientation preserving diffeomorphisms of $\Sigma$. Let ${\rm Diff}^0(\Sigma)\, \subset\, {\rm Diff}^+(\Sigma)$ be the
connected component containing the identity element. Then the quotient
$$
\Mod(\Sigma)\,:=\,\pi_0({\rm Diff}^+(\Sigma))\,=\, {\rm Diff}^+(\Sigma)/{\rm Diff}^0(\Sigma)
$$
is known as the mapping class group. Let ${\rm Com}(\Sigma)$ denote the space of $C^\infty$ complex structures on $\Sigma$
compatible with its orientation. The quotient ${\rm Teich}(\Sigma)\,:=\, 
{\rm Com}(\Sigma)/{\rm Diff}^0(\Sigma)$ for the natural action of ${\rm Diff}^0(\Sigma)$ on
${\rm Com}(\Sigma)$ is the Teichm\"uller space for $\Sigma$. The
group $\Mod(\Sigma)$ acts property discontinuously on ${\rm Teich}(\Sigma)$. The corresponding orbifold
$$
{\mathcal M}_g\,=\, {\rm Teich}(\Sigma)/\Mod(\Sigma)
$$
is the moduli space of smooth complex projective curves of genus $g$. This ${\mathcal M}_g$
is a smooth algebraic orbifold, and the underlying space is an irreducible complex quasiprojective 
variety of dimension $3(g-1)$.

In order to avoid stack issues, we often pass to a finite cover 
$\widehat{\Mcal}_g\,\longrightarrow\, \Mcal_g$ with smooth total space that is defined by a finite index
torsion-free subgroup of the mapping class group of genus $g$ (these exist \cite{Lo}, see below). It is well-known that 
$\widehat{\Mcal}_g$ then carries a universal family of curves, $\Xcal_{\widehat{\Mcal}_g}/\widehat{\Mcal}_g$.  
If we take the subgroup of the mapping class group to be normal, with quotient the finite group $G$, then this 
covering comes with a $G$-action. While the notions that we discuss below only pertain to families of smooth 
genus $g$ curves, we give them a sense to the universal curve by taking the corresponding $G$-invariant notion 
on $\Xcal_{\widehat{\Mcal}_g}/\widehat{\Mcal}_g$.

We consider families of projective curves $p\,:\,\Xcal\,\longrightarrow\, S$ whose fibers are smooth and 
geometrically irreducible. Let $\Pcal$ denote the functor which assigns to such a family the bundle of 
projective structures on its fibers; this is a torsor over $p_*K_{\Xcal/S}^{\otimes 2}$,
where $K_{\Xcal/S}^{\otimes 2}$ is the relative canonical bundle on $\Xcal$ for the projection $p$. We may compare this 
with the functor $\widetilde \Pcal$ which assigns to this family the isomorphism classes of pairs $(P,\,\widetilde D)$ 
consisting of a principal ${\rm PSL}(2,{\mathbb C})$--bundle on $\Xcal$ that are fiberwise topologically trivial 
endowed with a fiberwise algebraic connection $\widetilde D$ on $P$ that is (fiberwise) \emph{irreducible} in 
the sense that it is not induced by a connection on a principal $B$--bundle \cite{Si1}, \cite{Si2}. We here 
recall that the topological types of principal ${\rm PSL}(2,{\mathbb C})$--bundles on a smooth projective curve are 
parametrized by $\pi_1({\rm PSL}(2,{\mathbb C}))\,=\, {\mathbb Z}/2{\mathbb Z}$. The construction in the previous 
section identifies $\Pcal(\Xcal/S)$ with the closed subscheme of $\widetilde\Pcal(\Xcal/S)$ defined by the pairs 
$(P,\, \widetilde{D})$ for which $P$ coincides with $P^0$ in \eqref{e3}.

We shall abbreviate $\Pcal(\Xcal_{\widehat{\Mcal}_g}/\widehat{\Mcal}_g)$ by $\Pcal\big\vert\widehat\Mcal_g$ and we will 
simply write $\Pcal_g$ for its $G$-quotient (assuming that $\widehat{\Mcal}_g\,\longrightarrow\, \Mcal_g$ is a 
$G$-cover). So we have the forgetful morphisms
\begin{equation}\label{e4}
\Phi\,:\, \Pcal\big\vert\widehat\Mcal_g\,\longrightarrow\, \widehat\Mcal_g
\end{equation}
and $\Pcal_g\,\longrightarrow\, \Mcal_g$, the former being an affine
bundle over $\widehat\Mcal_g$ that is modeled on the cotangent bundle $T^*\widehat\Mcal_g$. This means that if we filter $\Phi_*\Ocal_{\Pcal|\widehat\Mcal_g}$ by fiberwise degree:
\begin{equation}\label{j4}
\Ocal_{\widehat\Mcal_g}\,=\,\Phi_*^{\le 0}\Ocal_{\Pcal|\widehat\Mcal_g}\,\subset
\,\Phi_*^{\le 1}\Ocal_{\Pcal|\widehat\Mcal_g}\,\subset\,\cdots\,\subset\, \Phi_*^{\le d}\Ocal_{\Pcal|\widehat\Mcal_g}
\,\subset\,\cdots,
\end{equation}
then this filtration exhausts $\Phi_*\Ocal_{\Pcal|\widehat\Mcal_g}$ (the latter is the union of these subsheaves which is same as 
the direct limit) and for $d\,\ge\, 1$, the successive quotient $\Phi_*^{\le d}\Ocal_{\Pcal|\widehat\Mcal_g}/\Phi_*^{\le 
d-1}\Ocal_{\Pcal|\widehat\Mcal_g}$ can be identified with the symmetric power $\sym^d T\widehat\Mcal_g$.

\section{Functions on the moduli space}

The forgetful morphism $\Pcal_g\,\longrightarrow\, \Mcal_g$ is affine. Since ${\mathcal M}_2$ is affine, it follows that 
$\Pcal_2$ is also affine. On the other hand, the only regular functions on $\Mcal_g$ for $g\ge 3$ are the constants. The 
main result of this section is that this does not change if we pass to $\Pcal_g$:

\begin{theorem}\label{thm1}
For $g\, \geq\, 3$, every regular function on $\Pcal_g$ is constant.
\end{theorem}

\begin{proof}
If $\widehat{\Mcal}_g\,\longrightarrow\, \Mcal_g$ is a smooth finite cover as above, then it of course suffices 
to show that $\Pcal\big\vert\widehat\Mcal_g$ has the stated property. The algebra of regular functions on this 
space is filtered as in \eqref{j4}. Recall that the filtration $\Phi_*^{\le d}\Ocal_{\Pcal 
|\widehat\Mcal_g}$ is exhaustive, $$H^0(\widehat\Mcal_g, \,\Phi_*^{\le 0})\,=\,H^0(\widehat\Mcal_g,\, 
\Ocal_{\widehat\Mcal_g})\,=\,\mathbb C$$ and that we have exact sequences
\[
0\,\longrightarrow\, H^0(\widehat\Mcal_g,\, \Phi_*^{\le d-1}\Ocal_{\Pcal |\widehat\Mcal_g})
\,\longrightarrow\,H^0(\widehat\Mcal_g, \,\Phi_*^{\le d}\Ocal_{\Pcal |\widehat\Mcal_g})\,\longrightarrow\,
H^0(\widehat\Mcal_g,\, \sym^dT\widehat\Mcal_g).
\]
So the theorem will follow from Proposition \ref{prop1} below.
\end{proof}

\begin{proposition}\label{prop1}
The sheaf $\sym^d(T\widehat{\Mcal}_g)$ has no nonzero sections when $d\,>\,0$.
\end{proposition}

Proposition \ref{prop1} will be proved in Section \ref{sep}

For any integer $l\, \geq\, 2$, a smooth curve of genus $g$ with level $l$ structure is an irreducible smooth projective
curve $S$ of genus $g$ together with a symplectic basis of $H^1(S,\, {\mathbb Z}/l{\mathbb Z})$ (see \cite{Pu}, \cite{Lo}).
Let $\Mcal_g[l]$ denote the moduli space of smooth curves of genus $g$ with level $l$ structure.

For the proof of Proposition \ref{prop1} we need the following lemma.

\begin{lemma}\label{lemma:abundance}
Let $l\, \geq\, 3$ be an integer, and let $\Mcal_g[l]\,\longrightarrow\, \Mcal_g$ be the finite cover
which parametrizes the smooth projective curves of genus $g$ endowed with a level $l$ structure.
Then any two points of $\Mcal_g[l]$ can be connected by a complete curve in $\Mcal_g$ whose irreducible
components are smooth.
\end{lemma}

\begin{proof}
We first prove that $\Mcal_g[l]$ admits a projective compactification with codimension two boundary, a fact that 
is well-known among experts. Let $\Acal_g[l]$ stand for the moduli space of principally polarized abelian 
varieties with full level $l$ structure. Both $\Mcal_g[l]$ and $\Acal_g[l]$ are known to be smooth \cite{Pu},
\cite{Lo}. The Torelli 
theorem asserts that the evident period map $P\,:\, \Mcal_g[l]\,\longrightarrow\,\Acal_g[l]$ identifies 
$\Mcal_g[l]$ with a locally closed subset of $ \Acal_g[l]$. According to Mumford (see \cite[Theorem 9.28]{Na}), this
map $P$ extends as a morphism from the Deligne-Mumford compactification $\overline{\Mcal}_g[l]$ of $\Mcal_g[l]$ to 
the minimal compactification $\Acal_g[l]^*$. This extended map essentially assigns to a stable curve the product the Jacobians 
of its irreducible components. So if we denote the closure of $\Mcal_g[l]$ in $\Acal_g[l]^*$ by $\Mcal_g[l]^*$ 
then the `minimal' boundary $\Mcal_g[l]^*\setminus \Mcal_g[l]$ is of codimension at least two in $\Mcal_g[l]^*$. We note that
the maximal dimension $3g-5\,=\,(3g-3)-2$ is in fact realized by the boundary component defined by one-point unions of a 
genus $(g-1)$-curve and a genus 1 curve. Since the minimal compactification $\Acal_g[l]^*$ is projective, so is 
$\Mcal_g[l]^*$.

We now assume $\Mcal_g[l]^*$ realized as a closed subset of some $\PP^N$. Given $p\,\in\,\Mcal_g[l]$, consider the 
Grassmannian $Gr_p$ of linear subspaces of $\PP^N$ of codimension $3g-2$ passing through $p$. By the
Bertini-Sard theorem (see \cite{Ve}), the subspaces which
\begin{itemize}
\item avoid the minimal boundary $\Mcal_g[l]^*\setminus \Mcal_g[l]$,
and

\item meet $\Mcal_g[l]$ transversally
\end{itemize}
make up an open dense subspace $Gr^\circ_p\,\subset\, Gr_p$. By
definition every member of $Gr^\circ_p$ is represented by a linear subspace $L$ for which $\Mcal_g[l]^*\bigcap
L$ is a smooth projective curve contained in $\Mcal_g[l]$ which passes through $p$. The union of these
curves is a constructible subset of $\Mcal_g[l]$ that is Zariski dense. Hence it contains a nonempty Zariski
open subset $U_p$ of $\Mcal_g[l]$. If $q\,\in\,\Mcal_g[l]$ is another point, then $U_p\bigcap U_q$ will be
nonempty and so if $r\,\in \,U_p\bigcap U_q$, then there exist smooth projective curves $C'$ and $C''$ in
$\Mcal_g[l]$ which contain respectively $\{p,\,r\}$ and $\{q,\,r\}$. Their union $C'\cup C''$ is then as desired.
\end{proof}

\section{Proof of proposition}\label{sep}

\begin{proof}[Proof of Proposition \ref{prop1}]
Let us first note that that if this vanishing property holds for a given cover $\widehat{\Mcal}_g$, then it holds 
for all intermediate smooth covers. We can therefore assume without loss of generality that $\widehat{\Mcal}_g$ 
dominates a cover of the form $\Mcal_g[l]$ for some $l\,\ge \,3$. The resulting covering map 
$\widehat{\Mcal}_g\,\longrightarrow\,\Mcal[l]$ is then \'etale and as $\widehat{\Mcal}_g$ is connected, it will 
then by Lemma \ref{lemma:abundance} have the same property as $\Mcal[l]$: any two points of $\widehat{\Mcal}_g$ lie 
on a connected complete curve whose irreducible components are smooth. Hence it suffices to show that for every 
smooth curve
\begin{equation}\label{ega}
\gamma\, :\, C\,\hookrightarrow\, \widehat{\Mcal}_g
\end{equation}
and every positive integer $d\,>\,0$, the vector 
bundle $\gamma^*\sym^d(T\widehat{\Mcal}_g)$ has no nonzero sections.

Let $p\,:\, \Xcal\, \longrightarrow\, C$ be the restriction of the universal family to $C$ in \eqref{ega}. The relative cotangent 
bundle $K_{\Xcal/C}$ is ample, for if we equip it with the fiberwise Poincar\'e metric, then according to 
\cite[p.~2, Main Theorem]{Sc} the curvature of the corresponding Chern connection is positive. Ampleness is an open 
property in the N\'eron-Severi group tensored with $\mathbb Q$ and so if we choose $x_0\in C$, then there is a positive 
integer $\ell_0$ such that $K_{\Xcal/C}^{\otimes \ell}\otimes p^*\Ocal_C(-2x_0)$ is ample for all $\ell\, \geq\, 
\ell_0$.

Choose a  finite map
$$
\pi\, :\, \widehat{C}\, \longrightarrow\, C
$$
of degree $\ell\, \geq\, \ell_0$ which has total ramification in $x_0$, and denote the unique point of $\widehat C$ 
over $x_0$ by $\widehat{x}_0$. This ensures that if we make the corresponding base change of $\Xcal/C$ over 
$\pi$,
$$
\begin{matrix}
\widehat{\Xcal}& \stackrel{\widetilde\pi}{\longrightarrow} & \Xcal\,\\
\,\,\, \Big\downarrow \widehat p && \, \,\,\,\Big\downarrow p \\
\widehat C & \stackrel{\pi}{\longrightarrow} & C
\end{matrix}
$$
the line bundle $\big(K_{\widehat{\Xcal}/\widehat C} \otimes \widehat{p}^*\Ocal_{\widehat C}(-2\widehat{x}_0)\big)^{\otimes l}$
is the pull-back  of an  ample  line bundle, namely $K_{\Xcal/C}^{\otimes l}\otimes  p^*\Ocal_{C}(-2x_0)$, along  the finite morphism $\widetilde\pi$, so that 
$K_{\widehat{\Xcal}/\widehat C} \otimes \widehat{p}^*\Ocal_{\widehat C}(-2\widehat{x}_0)$ is also ample (see for instance  \cite[p.~25, Proposition 4.4]{Ha}).
Then \cite[p.~75, Proposition 2.43]{Vi} says that the direct image
$$
\widehat{p}_*(K_{\widehat{\Xcal}/\widehat C}^{\otimes 2}\otimes \widehat{p}^*\Ocal_{\widehat C}(-2\widehat{x}_0))
\,=\, \widehat{p}_*K_{\widehat{\Xcal}/\widehat C}^{\otimes 2} \otimes \Ocal_{\widehat C}(-2\widehat{x}_0)
$$
on $\widehat C$ is a nef vector bundle. As $\Ocal_{\widehat C}(\widehat{x}_0)$ is ample, this implies that the 
vector bundle $\widehat{p}_* K^{\otimes 2}_{\widehat{\Xcal}/\widehat C}\otimes \Ocal_{\widehat C}(-\widehat{x}_0)$ 
is ample \cite[p.~360, Proposition 2.4]{Fu}. In particular, there exists an integer $r_0$ such that 
$\sym^{r}\widehat{p}_* K^{\otimes 2}_{\widehat{\Xcal}/\widehat C}\otimes \Ocal_{\widehat C}(-rx_0)$ is globally 
generated for $r\,\ge \,r_0$. But $\widehat{p}_* K^{\otimes 2}_{\widehat{\Xcal}/\widehat C}$ is the pull-back of 
$T^*\widehat{\Mcal}_g$ under the composite $\gamma\circ\pi$ (see \eqref{ega}) and so if $\sigma$ is section of 
$\gamma^*\sym^d(T\widehat{\Mcal}_g)$, then for every integer $s\,\ge\, 0$, its $s$-th power $\sigma^s$ defines a 
contraction homomorphism
\[
\iota_{\sigma^s}: \sym^{ds}\widehat{p}_* K^{\otimes 2}_{\widehat{\Xcal}/\widehat C}\otimes \Ocal_{\widehat C}(-ds\widehat{x}_0)\,\longrightarrow\, \Ocal_{\widehat C}(-ds\widehat{x}_0),
\]
which is nonzero when $\sigma$ is. If we take $s\,>\,\frac{r_0}{d}$, then the left hand side is globally generated, whereas the 
right hand side has no nonzero sections. Hence we conclude that $\sigma\,=\,0$.
\end{proof}

\section*{Acknowledgements}

The author is very grateful to Eduard Looijenga for sharing his ideas.


\end{document}